\documentclass[12pt, reqno]{amsart}
\usepackage{amsmath, amsthm, amscd, amsfonts, amssymb, graphicx, color}
\usepackage[bookmarksnumbered, colorlinks, plainpages]{hyperref}

\hypersetup{colorlinks=true,linkcolor=red, anchorcolor=green, citecolor=cyan, urlcolor=red, filecolor=magenta, pdftoolbar=true}

\newtheorem{theorem}{Theorem}[section]
\newtheorem{lemma}[theorem]{Lemma}

\newtheorem{corollary}[theorem]{Corollary}
\theoremstyle{definition}

\theoremstyle{remark}
\newtheorem{remark}[theorem]{Remark}
\numberwithin{equation}{section}

\newcommand{\be}{\begin{equation}}
\newcommand{\ee}{\end{equation}}

\newcommand{\cM}{{\mathcal M}}
\newcommand{\NN}{\mathbb{N}}

\begin{document}
\setcounter{page}{1}

\title[Monotonicity properties of weighted geometric means]{Monotonicity properties of Hadamard weighted geometric means of bounded sets of positive kernel operators}

\author{ Katarina Bogdanovi\'{c}$^{1}$, Aljo\v{s}a Peperko$^{2,3*}$}
\date{\thanks{
Faculty of Mathematics$^1$,
University of Belgrade,
Studentski trg 16,
SRB-11000 Belgrade, Serbia,
email:   katarinabgd77@gmail.com \\
  \\
Faculty of Mechanical Engineering$^2$,
University of Ljubljana,
A\v{s}ker\v{c}eva 6,
SI-1000 Ljubljana, Slovenia,\\
Institute of Mathematics, Physics and Mechanics$^3$,
Jadranska 19,
SI-1000 Ljubljana, Slovenia \\
e-mail:   aljosa.peperko@fs.uni-lj.si\\
*Corresponding author
} \today}





\begin{abstract}
We prove new monotonicity properties for joint and generalized spectral radius and their essential versions, 
operator norm, Hausdorff measure of non-compactness and numerical radius 
of products and sums of  Hadamard weighted geometric means of  positive kernel operators 
on appropriate Banach function spaces. 
Our results extend recent results  on weighted geometric symmetrizations of bounded sets of such operators defined on $L^2$. To our knowledge, several proved properties in this generality are new even in the finite dimensional case and also in the case of singelton sets. 
\end{abstract} \maketitle

\section{Introduction}

Let $A=[a_{ij}]$ be an entrywise nonnegative $n\times n$ matrix and let $S(A)=[\sqrt{a_{ij}a_{ji}}]$ be its geometric symmetrization. In \cite{Schwenk86}, Schwenk proved the inequality  
\be
r (S(A)) \le r(A),
\label{prva}
\ee 
for the spectral radius $r(\cdot)$ by using a graph-theoretical approach. In \cite{EJS88},  Elsner, Johnson and Dias Da Silva proved that the inequality 
\be
 r(A_1 ^{( \alpha _1)} \circ A_2 ^{(\alpha _2)} \circ \cdots \circ A_m ^{(\alpha _m)} ) \le
r(A_1)^{ \alpha _1} \, r(A_2)^{\alpha _2} \cdots r(A_m)^{\alpha _m} 
\label{gl1vecr_matrix}
\ee
for Hadamard weighted geometric mean holds for  nonnegative $n\times n$ matrices $A_1, A_2, \ldots , A_m$ and nonnegative numbers $\alpha _1$, $\alpha _2$,..., $\alpha _m$ such that $\sum_{j=1}^{m} \alpha _j \ge 1$. Here $A^{( \alpha)}=[a_{ij} ^{\alpha}]$ denotes the Hadamard (Schur) power of $A$ and $A\circ B=[a_{ij}b_{ij}]$ denotes the  Hadamard (Schur) product of matrices $A$ and $B$. Clearly, (\ref{gl1vecr_matrix}) generalizes (\ref{prva}), since  $S(A)= A^{( \frac{1}{2})} \circ (A^T) ^{(\frac{1}{2})}$. Let us point out that inequality (\ref{gl1vecr_matrix}) can straightforwardly be deduced from an earlier result by Kingman \cite{K61} and that (\ref{prva}) is a special case of earlier results by Karlin and Ost \cite[Theorem 2.1 and Remark 1]{KO85} and that the case $\sum_{j=1}^{m} \alpha _j = 1$ of (\ref{gl1vecr_matrix}) was already obtained in \cite[Remark 1]{KO85}. In \cite{KO85} these results were applied in the context of finite stationary Markov chains. Since then inequalities and equalities on  Hadamard weighted geometric means and weighted geometric symmetrizations received a lot of attention and have been applied in a variety of contexts (see e.g. \cite{
 D92, DP05, P06, Sh07, P19, 
 BP21, BP22b,  
LP24, 
BP24a, BP25} 
and also a more detailed list of references in \cite{BP24a} and \cite{BP25}. 

In \cite{D92}, Drnov\v{s}ek proved that in the case when $\sum_{j=1}^{m} \alpha _j = 1$, inequality (\ref{gl1vecr_matrix}) holds also for positive compact operators on Banach function spaces. In \cite{DP05}, Drnov\v{s}ek and the second author of the current article established that the compactness assumption can be removed and that analogous results hold also for operator norm and also for numerical radius on $L^2$. Further they proved additional results for products of Hadamard weighted geometric means (see Theorem  \ref{thbegin} below).   In \cite{P06}, the second author also showed that analogous results also hold for Hausdorff measure of non-compactness and for essential spectral radius on suitable Banach functions spaces (including $L^2$, see Theorem \ref{thbegin} below). In \cite{DP05} and  \cite{P06}, also generalizations of inequality (\ref{prva}) for products  and sums of geometric symmetrizations of positive kernel operators were proved.

In \cite{Sh07}, Shen and Huang studied weighted geometric symmetrizations $S_{\alpha} (A)=[a_{ij} ^{\alpha} a_{ji}^{1-\alpha}]$ for $\alpha \in [0,1]$ and for nonnegative $n\times n$ matrices. They showed that for a given square nonnegative matrix $A$ the function $\alpha \mapsto r (S_{\alpha} (A))$ is decreasing on $[0,\frac{1}{2}]$ and increasing on $[\frac{1}{2}, 1]$ (\cite[Theorem 3.3]{Sh07}). They also proved an analogous result for the operator (largest singular value) norm (\cite[Theorem 2.3]{Sh07}).  In \cite[Theorem 2.7]{BP21}, we  obtained an analogous result for the spectral radius, essential spectral radius, operator norm, Hausdorff measure of non-compactness and numerical radius of weighted geometric symmetrizations of a given positive kernel operator on $L^2$. In \cite{BP24a} we further extended a technique of Shen and Huang to obtain additional results.  For instance, it was proved in \cite[Corollary 2]{BP24a} 
 that also the function 
$\alpha \mapsto r (S_{\alpha}(A_1)S_{\alpha}(A_2))$ is decreasing on $[0, \frac{1}{2}]$ and increasing on $[\frac{1}{2}, 1]$, where $A_1$ and $A_2$ are positive kernel operators on $L^2$ (and that the analogue of this result holds also for the essential spectral radius). In \cite{BP25}, the results of \cite{BP24a} were generalized to the setting of bounded sets of positive kernel operators on  $L^2$ by establishing new results  for joint and generalized spectral radius, for essential versions of  joint and generalized spectral radius and for operator norm,  Hausdorff measure of non-compactness and numerical radius of such sets. Surprisingly to us, in the current article, we further generalize results from \cite{BP25} to general Hadamard weighted geometric means of bounded sets of positive kernel operators on more general Banach function spaces.

The article is organized in the following way. In Section 2 we recall some definitions and results that will be needed in our proofs. Introductory sections 1 and 2 are very similar to the ones in our article \cite{BP25} for the sake of completeness of information and for their presentation value on the topic studied.
 In Section 3 we prove new results. The main results of the article are Theorems \ref{S_alpha_simplifiedsets} and \ref{th2_jointgen}. 
To our knowledge also the obtained results for the singelton case (the monotonicity properties for weighted geometric means of positive kernel operators in Corollaries \ref{S_alpha_simplified} - \ref{naL2}, \ref{th2_jointgen3} and \ref{cor_simpler_sums_op}), are new.   

\section{Preliminaries}
\vspace{1mm}

Let $\mu$ be a $\sigma$-finite positive measure on a $\sigma$-algebra $\cM$ of subsets of a non-void set $X$.
Let $M(X,\mu)$ be the vector space of all equivalence classes of (almost everywhere equal)
complex measurable functions on $X$. A Banach space $L \subseteq M(X,\mu)$ is
called a {\it Banach function space} if $f \in L$, $g \in M(X,\mu)$,
and $|g| \le |f|$ imply that $g \in L$ and $\|g\| \le \|f\|$. Throughout the article, it is assumed that  $X$ is the carrier of $L$, that is, there is no subset $Y$ of $X$ of
 strictly positive measure with the property that $f = 0$ a.e. on $Y$ for all $f \in L$ (see \cite{Za83}).

If $\{f_n\}_{n\in \mathbb{N}} \subset M(X,\mu)$ is a decreasing real valued sequence and
$f=\inf\{f_n \in M(X,\mu): n \in \NN \}$, then we write $f_n \downarrow f$. 
A Banach function space $L$ has an {\it order continuous norm}, if $0\le f_n \downarrow 0$
implies $\|f_n\|_L \to 0$ as $n \to \infty$. 
In particular, we are interested in  Banach function spaces $L$ such that $L$ and its Banach dual space $L^*$ have order continuous norms. For standard examples of Banach function spaces  $L$, including such that $L$ and $L^*$ have order continuous norms, see e.g. \cite{BP24a, BP25, BS88,  BP21, BP22b, 
LP24} and the references cited there.

  
By an {\it operator} on a Banach function space $L$ we always mean a linear
operator on $L$.  An operator $K$ on $L$ is said to be {\it positive}
if it maps nonnegative functions to nonnegative ones, i.e., $KL_+ \subset L_+$, where $L_+$ denotes the positive cone $L_+ =\{f\in L : f\ge 0 \; \mathrm{a.e.}\}$.
Given operators $K$ and $H$ on $L$, we write $K \ge H$ if the operator $K - H$ is positive.

Recall that a positive  operator $K$ is always bounded, i.e., its operator norm
\be
\|K\|=\sup\{\|Kf\|_L : f\in L, \|f\|_L \le 1\}=\sup\{\|Kf\|_L : f\in L_+, \|f\|_L \le 1\}
\label{equiv_op}
\ee
is finite (the second equality in (\ref{equiv_op}) follows from $|Kf| \le K|f|$ for $f\in L$).
Also, its spectral radius $r(K)$ is always contained in the spectrum.

In the special case $L= L^2(X, \mu)$ we can define the {\it numerical radius} $w(K)$ of
a bounded operator $K$ on $L^2(X, \mu)$ by
$$ w(K) = \sup \{ | \langle K f, f \rangle | : f \in L^2(X, \mu), \| f \|_2 = 1 \} . $$
If, in addition, $K$ is positive, then it is easy to prove that
$$ w(K) = \sup \{ \langle K f, f \rangle  : f \in L^2(X, \mu)_+ , \| f \|_2 = 1 \} . $$
From this it follows easily that $w(K) \le w(H)$ for all positive operators $K$ and $H$ on $L^2(X, \mu)$ with $K \le H$.

An operator $K$ on a Banach function space $L$ is called a {\it kernel operator} if
there exists a $\mu \times \mu$-measurable function
$k(x,y)$ on $X \times X$ such that, for all $f \in L$ and for almost all $x \in X$,
$$ \int_X |k(x,y) f(y)| \, d\mu(y) < \infty \ \ \ {\rm and} \ \
   (Kf)(x) = \int_X k(x,y) f(y) \, d\mu(y)  .$$
One can check that a kernel operator $K$ is positive iff
its kernel $k$ is non-negative almost everywhere.

Let $L$ be a Banach function space such that $L$ and $L^*$ have order
continuous norms and let $K$ and $H$ be  positive kernel operators on $L$. By $\gamma (K)$ we denote the Hausdorff measure of
non-compactness of $K$, i.e.,
$$\gamma (K) = \inf\left\{ \delta >0 : \;\; \mathrm{there}\;\; \mathrm{is} \;\; \mathrm{a}\;\; \mathrm{finite}\;\; M \subset L \;\;\mathrm{such} \;\; \mathrm{that} \;\; K(D_L) \subset M + \delta D_L  \right\},$$
where $D_L =\{f\in L : \|f\|_L \le 1\}$. Then $\gamma (K) \le \|K\|$, $\gamma (K+H) \le \gamma (K) + \gamma (H)$, $\gamma(KH) \le \gamma (K)\gamma (H)$ and $\gamma (\alpha K) =\alpha \gamma (K)$ for $\alpha \ge 0$. Also
$0 \le K\le H$  implies $\gamma (K) \le \gamma (H)$ (see e.g. \cite[Corollary 4.3.7 and Corollary 3.7.3]{Me91}). Let $r_{ess} (K)$ denote the essential spectral radius of $K$, i.e., the spectral radius of the Calkin image of $K$ in the Calkin algebra. Then
\be
 r_{ess} (K) =\lim _{j \to \infty} \gamma (K^j)^{1/j}=\inf _{j \in \NN} \gamma (K^j)^{1/j}
\label{esslim=inf}
\ee
and $r_{ess} (K) \le \gamma (K)$. Recall that if $L=L^2(X, \mu)$, then $\gamma (K^*) = \gamma (K)$ and $r_{ess} (K^*)=r_{ess} (K)  $, where $K^*$ denotes the adjoint of $K$ (see e.g. \cite[Proposition 4.3.3, Theorems 4.3.6 and 4.3.13 and Corollary 3.7.3]{Me91}, \cite[Theorem 1]{Nussbaum70}, \cite{LP24}). Note that equalities (\ref{esslim=inf}) and  $r_{ess} (K^*)=r_{ess} (K)  $ are valid for any bounded operator $K$ on a given complex Banach space $L$ (see e.g. \cite[Theorem 4.3.13 and Proposition 4.3.11]{Me91}, \cite[Theorem 1]{Nussbaum70}) and that $\gamma (K^*) = \gamma (K)$ is valid for any bounded operator $K$ on a given complex Hilbert space $L$ (\cite{LP24}).

It is well-known that kernel operators play an important, often even central, role in a variety of applications from differential and integro-differential equations, problems from physics
(in particular from thermodynamics), engineering, statistical and economic models, etc (see e.g. \cite{J82, P19} 
and the references cited there).
For the theory of Banach function spaces and more general Banach lattices we refer the reader to the books \cite{Za83, BS88, AA02, AB85, Me91}.

\bigskip

Let $\Sigma$ be a bounded set of bounded operators on a complex Banach space $L$.
For $m \ge 1$, let
$$\Sigma ^m =\{A_1A_2 \cdots A_m : A_i \in \Sigma\}.$$
The generalized spectral radius of $\Sigma$ is defined by
\be
r (\Sigma)= \limsup _{m \to \infty} \;[\sup _{A \in \Sigma ^m} r (A)]^{1/m}
\label{genrho}
\ee
and is equal to
$$r (\Sigma)= \sup _{m \in \NN} \;[\sup _{A \in \Sigma ^m} r (A)]^{1/m}.$$
The joint spectral radius of $\Sigma$ is defined by
\be
\hat{r}  (\Sigma)= \lim _{m \to \infty}[\sup _{A \in \Sigma ^m} \|A\|]^{1/m}.
\label{BW}
\ee
Similarly, the generalized essential spectral radius of $\Sigma$ is defined by
\be
r_{ess} (\Sigma)= \limsup _{m \to \infty} \;[\sup _{A \in \Sigma ^m} r_{ess} (A)]^{1/m}
\label{genrhoess}
\ee
and is equal to
$$r_{ess} (\Sigma)= \sup _{m \in \NN} \;[\sup _{A \in \Sigma ^m} r_{ess} (A)]^{1/m}.$$
The joint essential  spectral radius of $\Sigma$ is defined by
\be
\hat{r} _{ess}  (\Sigma)= \lim _{m \to \infty}[\sup _{A \in \Sigma ^m} \gamma (A)]^{1/m}.
\label{jointess}
\ee

It is well known that $r (\Sigma)= \hat{r}  (\Sigma)$ for a precompact nonempty set $\Sigma$ of compact operators on $L$ (see e.g. \cite{ShT00, ShT08, Mo}),
in particular for a bounded set of complex $n\times n$ matrices (see e.g. 
 \cite{SWP97, Dai11, MP12} and the references cited there).
This equality is called the Berger-Wang formula or also the
generalized spectral radius theorem. 
It is known that also the generalized Berger-Wang formula holds, i.e, that for any precompact nonempty  set $\Sigma$ of bounded operators on $L$ we have
$$\hat{r}  (\Sigma) = \max \{r (\Sigma), \hat{r} _{ess}  (\Sigma)\}$$
(see e.g.  \cite{ShT08, Mo, ShT00}). Observe also that it was proved in \cite{Mo} that in the definition of  $\hat{r} _{ess}  (\Sigma)$ one may replace the Hausdorff measure of noncompactness by several other seminorms, for instance it may be replaced by the essential norm.

In general $r (\Sigma)$ and $\hat{r}  (\Sigma)$ may differ even in the case of a bounded set $\Sigma$ of compact positive operators on $L$ (see \cite{SWP97} or also \cite{P17}).
Also, in \cite{Gui82} the reader can find an example of two positive non-compact weighted shifts $A$ and $B$ on $L=l^2$ such that $\rho(\{A,B\})=0 < \hat{\rho}(\{A,B\})$. As already noted in \cite{ShT00} also $r_{ess} (\Sigma)$ and $\hat{r} _{ess}  (\Sigma)$ may in general be different.

The theory of the generalized and the joint spectral radius has many important applications for instance to discrete and differential inclusions,
wavelets, invariant subspace theory
(see e.g. 
\cite{Dai11, Wi02, ShT00, ShT08} and the references cited there).
In particular, $\hat{\rho} (\Sigma)$ plays a central role in determining stability in convergence properties of discrete and differential inclusions. In this
theory the quantity $\log \hat{\rho} (\Sigma)$ is known as the maximal Lyapunov exponent (see e.g. \cite{Wi02}). Note that versions of Berger-Wang formula hold also in appropriate non-linear settings (\cite{MP12, DGT25, LP26}).

We have 
the following well known facts that hold for all $\rho \in \{r,  \hat{r}, r_{ess}, \hat{r} _{ess}  \}$:
\be
\rho (\Sigma  ^m) = \rho (\Sigma)^m \;\;\mathrm{and}\;\;
\rho (\Psi \Sigma) = \rho (\Sigma\Psi)
\label{again}
\ee
where $\Psi \Sigma =\{AB: A\in \Psi, B\in \Sigma\}$ and $m\in \NN$.

Let $K$ and $H$ be positive kernel operators on a Banach function space $L$ with kernels $k$ and $h$ respectively,
and $\alpha \ge 0$.
The \textit{Hadamard (or Schur) product} $K \circ H$ of $K$ and $H$ is the kernel operator
with kernel equal to $k(x,y)h(x,y)$ at point $(x,y) \in X \times X$ which can be defined (in general)
only on some order ideal of $L$. Similarly, the \textit{Hadamard (or Schur) power}
$K^{(\alpha)}$ of $K$ is the kernel operator with kernel equal to $(k(x, y))^{\alpha}$
at point $(x,y) \in X \times X$ which can be defined only on some order ideal of $L$.

Let $K_1 ,\ldots, K_m$ be positive kernel operators on a Banach function space $L$,
and $\alpha _1, \ldots, \alpha _m$ nonnegative numbers such that $\sum_{j=1}^m \alpha _j = 1$.
Then the {\it  Hadamard weighted geometric mean}
$K = K_1 ^{( \alpha _1)} \circ K_2 ^{(\alpha _2)} \circ \cdots \circ K_m ^{(\alpha _m)}$ of
the operators $K_1 ,\ldots, K_m$ is a positive kernel operator defined
on the whole space $L$, since $K \le \alpha _1 K_1 + \alpha _2 K_2 + \ldots + \alpha _m K_m$ by the inequality between the weighted arithmetic and geometric means.

%

 Let us recall  the following result, which was proved in \cite[Theorem 2.2]{DP05} and
\cite[Theorem 5.1 and Example 3.7]{P06} (see also e.g. \cite[Theorem 2.1]{P17}).

\begin{theorem}
\label{thbegin}
Let $\{K_{i j}\}_{i=1, j=1}^{n, m}$ be positive kernel operators on a Banach function space $L$ and  $\alpha _1$, $\alpha _2$,..., $\alpha _m$  nonnegative numbers.

If 
$\sum_{j=1}^{m} \alpha _j = 1$, then the positive kernel operator
\be
K:= \left(K_{1 1}^{(\alpha _1)} \circ \cdots \circ K_{1 m}^{(\alpha _m)}\right) \ldots \left(K_{n 1}^{(\alpha _1)} \circ \cdots \circ K_{n m}^{(\alpha _m)} \right)
\label{osnovno}
\ee
satisfies the following inequalities
\begin{equation}
K \le
(K_{1 1} \cdots  K_{n 1})^{(\alpha _1)} \circ \cdots
\circ (K_{1 m} \cdots K_{n m})^{(\alpha _m)} , \\
\label{norm2}
\end{equation}
\begin{eqnarray}
\nonumber
\left\|K \right\| &\le &\left\|(K_{1 1} \cdots  K_{n 1})^{(\alpha _1)} \circ \cdots
\circ (K_{1 m} \cdots K_{n m})^{(\alpha _m)} \right\|\\
&\le&\left\|K_{1 1} \cdots  K_{n 1}\right\|^{\alpha _1}\cdots\left\|K_{1 m} \cdots  K_{n m}\right\|^{\alpha _m}
\label{spectral2}
\end{eqnarray}
\begin{eqnarray}
\nonumber
r\left(K \right)& \le &r\left((K_{1 1} \cdots  K_{n 1})^{(\alpha _1)} \circ \cdots
\circ (K_{1 m} \cdots A_{n m})^{(\alpha _m)}\right)\\
&\le&r\left( K_{1 1} \cdots  K_{n 1} \right)^{\alpha _1} \cdots
r\left( K_{1 m} \cdots K_{n m}\right)^{\alpha _m} .
\label{tri}
\end{eqnarray}
If, in addition, $L$ and $L^*$ have order continuous norms, then
\begin{eqnarray}
\nonumber
\gamma (K) & \le &\gamma \left((K_{1 1} \cdots  K_{n 1})^{(\alpha _1)} \circ \cdots
\circ (K_{1 m} \cdots K_{n m})^{(\alpha _m)}\right)\\
 &\le &
 \label{meas_noncomp}
\gamma (K_{1 1} \cdots  K_{n 1})^{\alpha _1} \cdots \gamma(K_{1 m} \cdots K_{n m})^{\alpha _m}, \\
\nonumber
r_{ess} \left(K \right) & \le &r_{ess} \left((K_{1 1} \cdots  K_{n 1})^{(\alpha _1)} \circ \cdots
\circ (K_{1 m} \cdots K_{n m})^{(\alpha _m)}\right)\\
&\le  &
\label{ess_spectral}
r_{ess} \left( K_{1 1} \cdots  K_{n 1} \right)^{\alpha _1} \cdots
r_{ess} \left( K_{1 m} \cdots K_{n m}\right)^{\alpha _m} .
\end{eqnarray}
If, in addition, $L=L^2(X, \mu)$, then 
\begin{eqnarray}
\nonumber
w (K) & \le &w \left((K_{1 1} \cdots  K_{n 1})^{(\alpha _1)} \circ \cdots
\circ (K_{1 m} \cdots K_{n m})^{(\alpha _m)}\right)\\
 &\le &
 \label{num_rad}
w(K_{1 1} \cdots  K_{n 1})^{\alpha _1} \cdots w(K_{1 m} \cdots K_{n m})^{\alpha _m}.
\end{eqnarray}
\label{DBPfs}
\end{theorem}

We will need the following well-known inequalities (see e.g. \cite{Mi}). For
nonnegative measurable functions and for nonnegative numbers $\alpha$ and $\beta$  such that $\alpha+\beta \ge 1$ we have
\be
f_1^{\alpha}g_1^{\beta}+\cdots+f_m^{\alpha}g_m^{\beta} \le (f_1+\cdots+f_m)^{\alpha}(g_1+\cdots+g_m)^{\beta}.
\label{mitrn} 
\ee
More generally, for nonnegative measurable functions $\{f _{i j}\}_{i=1, j=1}^{n, m}$ and for nonnegative numbers  $\alpha_j$, $j=1, \ldots , m$,   such that $\sum _{j=1} ^m \alpha _j \ge 1$ we have
\be
(f_{11}^{\alpha _1} \cdots f_{1m}^{\alpha _m})+\cdots+(f_{n1}^{\alpha _1} \cdots f_{nm}^{\alpha _m}) \le (f_{11}+\cdots+f_{n1})^{\alpha _1} \cdots (f_{1m}+\cdots+f_{nm})^{\alpha _m}.
\label{mitr2}
\ee

Let $\Psi _1, \ldots , \Psi _m$ be bounded sets of positive kernel operators on a Banach function space $L$ and let $\alpha _1, \ldots , \alpha _m$ be positive numbers such that
$\sum _{i=1} ^m \alpha _i = 1$. Then the bounded set of positive kernel operators on $L$, defined by
$$\Psi _1 ^{( \alpha _1)} \circ \cdots \circ \Psi _m ^{(\alpha _m)}=\{ A_1 ^{( \alpha _1)} \circ \cdots \circ A _m ^{(\alpha _m)}: A_1\in \Psi _1, \ldots, A_m \in \Psi _m \},$$
is called the {\it weighted Hadamard (Schur) geometric mean} of sets $\Psi _1, \ldots , \Psi _m$. The set
$\Psi _1 ^{(\frac{1}{m})} \circ \cdots \circ \Psi _m ^{(\frac{1}{m})}$ is called the  {\it Hadamard (Schur) geometric mean} of sets $\Psi _1, \ldots , \Psi _m$.

Let us denote $\|\Psi\|=\sup _{A\in\Psi} \|A\|$, $\gamma (\Psi)=\sup _{A\in\Psi} \gamma (A)$ and  $w (\Psi)=\sup _{A\in\Psi} w(A)$. 
The following result was established in \cite[Theorems 3.2(i) and 3.6(i)]{BP22b} and in \cite[Corollary 6]{BP25} by applying Theorem \ref{thbegin} and (\ref{mitr2}).

\begin{theorem} \label{finally_1}
Let $\{\Psi _{i j}\}_{i=1, j=1}^{k, m}$ be bounded sets of positive kernel operators on a Banach function space $L$,  let $\rho \in \{r, \hat{r}, \|\cdot\| \}$ and assume that
 $\alpha _1, \ldots , \alpha _m$ are positive numbers such that  $\sum _{i=1} ^m \alpha _i = 1$. 

Then
\begin{eqnarray}
\nonumber
& & \rho \left(\left(\Psi_{1 1}^{(\alpha _1)} \circ \cdots \circ \Psi_{1 m}^{(\alpha _m)}\right) \ldots \left(\Psi_{k 1}^{(\alpha _1)} \circ \cdots \circ \Psi_{k m}^{(\alpha _m)} \right) \right) \\
\nonumber
& \le &\rho \left((\Psi_{1 1} \cdots  \Psi_{k 1})^{(\alpha _1)} \circ \cdots
\circ (\Psi_{1 m} \cdots \Psi_{k m})^{(\alpha _m)}\right)\\
&\le&\rho \left( \Psi_{1 1} \cdots  \Psi_{k 1} \right)^{\alpha _1} \cdots
\rho \left( \Psi_{1 m} \cdots \Psi_{k m}\right)^{\alpha _m} 
\label{lepa_nwg}
\end{eqnarray}
and
\begin{eqnarray}
\nonumber
& & \rho \left(\left(\Psi_{1 1}^{(\alpha _1)} \circ \cdots \circ \Psi_{1 m}^{(\alpha _m)}\right) + \ldots + \left(\Psi_{k 1}^{(\alpha _1)} \circ \cdots \circ \Psi_{k m}^{(\alpha _m)} \right) \right) \\
\nonumber
& \le &\rho \left((\Psi_{1 1}+ \cdots  + \Psi_{k 1})^{(\alpha _1)} \circ \cdots
\circ (\Psi_{1 m} + \cdots  + \Psi_{k m})^{(\alpha _m)}\right)\\
&\le&\rho \left( \Psi_{1 1} + \cdots  + \Psi_{k 1} \right)^{\alpha _1} \cdots
\rho \left( \Psi_{1 m} + \cdots  + \Psi_{k m}\right)^{\alpha _m} .
\label{lepa2_nwg}
\end{eqnarray}

If, in addition, $L$ and $L^*$ have order continuous norms, then (\ref{lepa_nwg}) and (\ref{lepa2_nwg}) hold also for $\rho \in \{ r_{ess}, \hat{r}_{ess}, \gamma \}$.

If, in addition, $L=L^2(X, \mu)$, then (\ref{lepa_nwg}) and (\ref{lepa2_nwg}) hold also for $\rho=w$.


\end{theorem}

Let $K$ be a positive kernel operator on $L=L^2(X, \mu)$ with a kernel $k$ and let $\alpha \in [0,1]$.
Denote by $S_{\alpha} (K) = K^{(\alpha )} \circ (K^*)^{(1-\alpha)} $ a positive kernel operator on $L$
 with a kernel $s_{\alpha} (k)(x,y)=k^{\alpha}(x,y)k^{1-\alpha}(y,x)$. Note that $S (K) =S_{\frac{1}{2}} (K) $ is a geometric symmetrization of $K$, which is a selfadjoint and positive kernel operator on  $L^2(X, \mu)$ with a kernel $\sqrt{k(x,y)k(y,x)}$.

Let $\Psi$ be a bounded set of positive kernel operators on $L=L^2(X, \mu)$ and $\alpha\in [0, 1]$. Denote by $\Psi^*$ and $S_{\alpha} (\Psi)$ bounded sets of positive kernel operators on $L$ defined by $\Psi^*=\{A^* : A\in \Psi\}$ and $S_{\alpha} (\Psi)=\Psi^{(\alpha)} \circ (\Psi^*)^{(1-\alpha)}=\{A^{(\alpha)}\circ (B^*)^{(1-\alpha)} : A, B \in \Psi \}$. Note that $S_{\alpha} (\Psi)^*=S_{\alpha} (\Psi^*)=\{(A^*)^{(\alpha)}\circ B^{(1-\alpha)} : A, B \in \Psi\}$.

Let $\Psi, \Psi_1, \ldots , \Psi_n$ be bounded sets of positive kernel operators on $L=L^2(X, \mu)$, $\rho \in \{r, \hat{r}, r_{ess}, \hat{r}_{ess}, \gamma, \|\cdot\|, w\}$ and $\alpha \in [0,1]$.
In \cite[Proposition 4.2]{BP22b} and \cite[Corollary 7]{BP25} it was proved that
$$\rho (S_{\alpha} (\Psi_{1}) \cdots  S_{\alpha}(\Psi_{n}))\;\;\;\;\;\;\;\;\;\;\;\;\;\;\;\;\;\;\;\;\;\;\;\;\;\;\;\;\;\;\;\;\;\;\;\;\;\;\;\;\;\;\;\;\;\;\;\;\;\;\;\;\;\;\;\;\;\;\;\;\;\;\;\;\;\;\;\;$$
\be\le \rho \left((\Psi_1 \cdots \Psi_n )^{(\alpha)} \circ ((\Psi_n \cdots \Psi_1)^*)^{(1-\alpha)}  \right) \le \rho (\Psi_1 \cdots \Psi_n )^{\alpha} \, \rho (\Psi_n \cdots \Psi_1 )^{1-\alpha}
\label{bauernprotest}
\ee
and
\begin{eqnarray}
\nonumber
& & \rho (S_{\alpha}(\Psi_1)+ \cdots + S_{\alpha}(\Psi_m)) \le \rho \left(S_{\alpha}(\Psi_1+ \cdots+ \Psi_m)\right) \\
& \le & \rho (\Psi_1+ \cdots+ \Psi_m).
\label{widerstand}
\end{eqnarray}
In particular, 
\be
\label{eden}
\rho (S_{\alpha} (\Psi) ) \le \rho (\Psi ).
\ee

The following results were the main results of \cite{BP25} (\cite[Theorems 8 and 14]{BP25}).
\begin{theorem}
\label{thrun}
Let $\Psi_1, \ldots , \Psi_n$ be bounded sets of positive kernel operators on $L=L^2(X, \mu)$ and $\rho \in \{r, \hat{r}, r_{ess}, \hat{r}_{ess}, \|\cdot\|, \gamma, w \}$. 
Define $\rho_n: [0,1] \to [0,\infty)$ by
$$\rho_n (\alpha)=\sqrt{\rho(S_{\alpha} (\Psi_1)S_{\alpha} (\Psi_2)\cdots S_{\alpha} (\Psi_n))\rho(S_{\alpha} (\Psi_n)S_{\alpha} (\Psi_{n-1})\cdots S_{\alpha} (\Psi_1))}.$$
Then $\rho_n$ is decreasing on $[0, \frac{1}{2}]$ and increasing on $[\frac{1}{2}, 1]$. 

In particular, $\rho_n(\alpha) \ge \rho_n (\frac{1}{2})$ for each $\alpha \in [0,1]$.
\end{theorem}

\begin{theorem}
\label{th2_joint}
Let $\Psi_{ij}$ for $i=1, \ldots , n$ and $j=1, \ldots , m$ be bounded sets of positive kernel operators on $L^2(X, \mu)$ and $\rho \in \{r, \hat{r}, r_{ess}, \hat{r}_{ess}, \|\cdot\|, \gamma, w \}$. Define $\overline{\rho} _n: [0,1] \to [0, \infty)$ by
$$\overline{\rho}_n (\alpha)=\left(\rho((S_{\alpha} (\Psi_{11}) + \cdots + S_{\alpha}(\Psi_{1m}))  \cdots (S_{\alpha} (\Psi_{n1}) + \cdots + S_{\alpha}(\Psi_{nm})))\right)^{\frac{1}{2}}\times $$
$$ \left(\rho( (S_{\alpha} (\Psi_{n1}) + \cdots + S_{\alpha}(\Psi_{nm}))\cdots (S_{\alpha} (\Psi_{11}) + \cdots + S_{\alpha}(\Psi_{1m}))) \right)^{\frac{1}{2}}.$$
 Then
\be
\overline{\rho} _n (\alpha) \le \rho (\left (\Psi_{11} + \cdots + \Psi_{1m})\cdots (\Psi_{n1} + \cdots + \Psi_{nm}) \right)^{\frac{1}{2}} \times
\label{cute} 
\ee
$$\rho (\left (\Psi_{n1} + \cdots + \Psi_{nm})\cdots (\Psi_{11} + \cdots + \Psi_{1m}) \right)^{\frac{1}{2}} $$
for each $\alpha \in [0,1]$. Moreover, $\overline{\rho}_n$ is decreasing on $[0,\frac{1}{2}]$ and increasing on $[\frac{1}{2}, 1]$.
\end{theorem}

In the next section we generalize Theorems \ref{thrun} and \ref{th2_joint} 
 to the setting of 
Hadamard weighted geometric means of bounded sets of positive kernel operators on Banach function spaces.

\section{Monotonicity properties of Hadamard weighted geometric means of bounded sets}


For two positive kernel operators $K, H$ on a Banach function space $L$ 
 and for $\alpha\in [0, 1]$ let us denote the Hadamard weighted geometric mean of $K$ and $H$ by $S_{\alpha} (K, H)=K^{(\alpha)}\circ H^{(1-\alpha)}$. If, in particular, $L=L^2(X, \mu)$, then  let us denote the Hadamard weighted geometric mean of $K$ and $H^*$ by $T_{\alpha} (K, H)=K^{(\alpha)}\circ (H^*)^{(1-\alpha)}$. 
 More generally, for two bounded sets $\Psi$ and $\Sigma$ of  positive kernel operators  on a Banach function space $L$ 
 and for $\alpha\in [0, 1]$ let us denote the Hadamard weighted geometric mean of $\Psi$ and $\Sigma$ by $S_{\alpha} (\Psi, \Sigma)=\Psi^{(\alpha)}\circ \Sigma^{(1-\alpha)}=\{A^{(\alpha)}\circ B^{(1-\alpha)}: A \in \Psi,\, B\in \Sigma \}=S_{1-\alpha} (\Sigma , \Psi)$. If, in particular, $L=L^2(X, \mu)$, then  let us denote the Hadamard weighted geometric mean of $\Psi$ and $\Sigma^*$ by $T_{\alpha} (\Psi, \Sigma)=\Psi^{(\alpha)}\circ (\Sigma^*)^{(1-\alpha)}=\{A^{(\alpha)}\circ (B^*)^{(1-\alpha)}: A\in \Psi, B\in \Sigma\}$. Then $T_{\alpha} (\Psi, \Sigma)=S_{\alpha} (\Psi, \Sigma^*)$ and $T_{\alpha} (\Psi, \Psi)=S_{\alpha} (\Psi, \Psi^*) =  S_{\alpha} (\Psi)$. Note also that $S_{\alpha} (\Psi, \Sigma)^*=\{(A^*)^{(\alpha)}\circ (B^*)^{(1-\alpha)}: A\in \Psi, B\in \Sigma\}=S_{\alpha}(\Psi ^*, \Sigma^*)$ and 
 $T_{\alpha} (\Psi, \Sigma)^*=\{(A^*)^{(\alpha)}\circ B^{(1-\alpha)}: A\in \Psi, B\in \Sigma\}=T_{\alpha}(\Psi ^*, \Sigma^*) =T_{1-\alpha} (\Sigma, \Psi)$. 
 Observe also that 
 \begin{equation}
 \label{*prod}
 (\Psi \Sigma)^*=\Sigma^* \Psi ^* \;\; \mathrm{and}\;\; \rho (\Psi ^*) = \rho (\Psi) 
 \end{equation}
 for each $\rho \in \{ r, \hat{r},  r_{ess}, \hat{r}_{ess}, \gamma, \|\cdot\|, w \}$.
 The following result follows easily from Theorem \ref{finally_1}. 
\begin{corollary}
\label{needed1sets}
 Let $\Psi_1, \ldots , \Psi_n, \Sigma_1, \ldots , \Sigma_n$ be bounded sets of positive kernel operators on a Banach function space $L$, $\alpha\in [0, 1]$ and 
assume that $\rho \in \{ r, \hat{r}, \|\cdot\|\}$.  Then 
\be
\rho (S_{\alpha} (\Psi_{1}, \Sigma_1) \cdots  S_{\alpha}(\Psi_{n}, \Sigma_n))
\label{aufstand2dvabrezsets}
\ee
$$\le \rho \left(S_{\alpha} (\Psi_1 \cdots \Psi_n , \Sigma_1 \cdots \Sigma_n )  \right) 
\le \rho (\Psi_1 \cdots \Psi_n )^{\alpha} \, \rho (\Sigma_1 \cdots \Sigma_n )^{1-\alpha}$$
and
\be
\rho(S_{\alpha}(\Psi_1 ,  \Sigma_1)+\cdots +S_{\alpha}(\Psi_n , \Sigma_n)) 
\label{20_2brezsets}
\ee
$$\le  \rho (S_{\alpha}(\Psi_1+ \cdots + \Psi_n ,  \Sigma_1 + \cdots + \Sigma_n))  \le \rho (\Psi_1+ \cdots + \Psi_n )^{\alpha} \rho (\Sigma_1+ \cdots + \Sigma_n )^{1-\alpha}.$$

If, in addition $L$ and $L^*$ have order continuous norm, then Inequalities (\ref{aufstand2dvabrezsets})  and (\ref{20_2brezsets}) hold also for $\rho \in \{r_{ess}, \hat{r}_{ess},\gamma \}$. If, in addition,  $L=L^2(X, \mu)$ then Inequalities (\ref{aufstand2dvabrezsets})  and (\ref{20_2brezsets}) hold also for $\rho = w$.
\end{corollary}
Similarly, the following result follows easily from Theorem \ref{finally_1} and (\ref{*prod}).
\begin{corollary}
Let $\rho \in \{ r, \hat{r}, r_{ess} , \hat{r}_{ess}, \gamma, \|\cdot\|, w \}$,  $\alpha\in [0, 1]$ and  $\Psi_1, \ldots , \Psi_n, \Sigma_1, \ldots , \Sigma_n$ be bounded sets of positive kernel operators on $L=L^2(X, \mu)$. Then 
\be
\rho (T_{\alpha} (\Psi_{1}, \Sigma_1) \cdots  T_{\alpha}(\Psi_{n}, \Sigma_n))
\label{aufstand2dvasets}
\ee
$$\le \rho \left(T_{\alpha} (\Psi_1 \cdots \Psi_n , \Sigma_n \cdots \Sigma_1 )  \right) 
\le \rho (\Psi_1 \cdots \Psi_n )^{\alpha} \, \rho (\Sigma_n \cdots \Sigma_1 )^{1-\alpha}$$
and
\be
\rho(T_{\alpha}(\Psi_1 ,  \Sigma_1)+\cdots +T_{\alpha}(\Psi_n , \Sigma_n)) 
\label{20_2sets}
\ee
$$\le  \rho (T_{\alpha}(\Psi_1+ \cdots + \Psi_n ,  \Sigma_1 + \cdots + \Sigma_n))  \le \rho (\Psi_1+ \cdots + \Psi_n )^{\alpha} \rho (\Sigma_1+ \cdots + \Sigma_n )^{1-\alpha}.$$
\end{corollary}

 The following result generalizes Theorem \ref{thrun}. 

\begin{theorem}
\label{S_alpha_simplifiedsets}
For bounded sets  $\Psi_1, \ldots , \Psi_n, \Sigma_1, \ldots , \Sigma_n$ of positive kernel operators on  a Banach function space $L$, 
$\alpha \in [0,1]$ and for $\rho \in \{ r,  \hat{r}, \|\cdot\|\}$ 
define 
$$\tilde{\rho} _{n}(\alpha):=\sqrt{\rho(S_{\alpha}(\Psi_1, \Sigma_1)\cdots S_{\alpha}(\Psi_n, \Sigma_n))
\rho(S_{\alpha}(\Sigma_1, \Psi_1)\cdots S_{\alpha}(\Sigma_n, \Psi_n))}.$$
 Then $\tilde{\rho}_{n}$ is decreasing on $[0, \frac{1}{2}]$ and increasing on $[\frac{1}{2}, 1]$. In particular,  $\tilde{\rho} _{n}(\alpha) \ge \tilde{\rho} _{n}(\frac{1}{2})$ for all $\alpha \in [0,1]$.
 
 If, in addition $L$ and $L^*$ have order continuous norms, then the same conclusions hold also for $\rho \in \{r_{ess}, \hat{r}_{ess}, \gamma \}$ and if, in addition, $L=L^2(X, \mu)$ then  the same conclusions hold also for $\rho = w$.
\end{theorem}
\begin{proof} 
Let $\rho \in \{ r, \hat{r}, \|\cdot\|\}$ and let $\Psi$ and $\Sigma$ be bounded sets of positive kernel operators on  a Banach function space $L$ . First consider the case when $0\le\alpha_1<\alpha_2\le\frac{1}{2}$. Let $\alpha=\frac{\alpha_1+\alpha_2-1}{2\alpha_1-1} \in (0, 1)$.  Then  
\be
S_{\alpha_2}(\Psi, \Sigma) \subset S_{\alpha}(S_{\alpha_1}(\Psi, \Sigma), S_{\alpha_1}(\Sigma, \Psi)).
\label{good1brezset}
\ee
Indeed, for each $K\in \Psi$ and $H \in \Sigma$ we have
$$S_{\alpha_2}(K, H)=K^{(\alpha_2)}\circ H^{(1-\alpha_2)}=K^{(\alpha\alpha_1+(1-\alpha_1)(1-\alpha))}\circ H^{(\alpha(1-\alpha_1)+\alpha_1(1-\alpha))}$$
$$=(K^{(\alpha_1)}\circ H^{(1-\alpha_1)})^{(\alpha)}\circ (H^{(\alpha_1)}\circ K^{(1-\alpha_1)})^{(1-\alpha)} $$
\be
=S_{\alpha}(S_{\alpha_1}(K, H), S_{\alpha_1}(H, K)) \in  S_{\alpha}(S_{\alpha_1}(\Psi, \Sigma), S_{\alpha_1}(\Sigma, \Psi)),
\label{good1brez}
\ee
which proves (\ref{good1brezset}).
  By (\ref{good1brezset}) and (\ref{aufstand2dvabrezsets}) we obtain
$$\tilde{\rho}_{n}(\alpha_2)=\sqrt{\rho(S_{\alpha_2}(\Psi_1,  \Sigma_1)\cdots S_{\alpha_2}(\Psi_n,  \Sigma_n))
\rho(S_{\alpha_2}( \Sigma_1, \Psi_1)\cdots S_{\alpha_2}( \Sigma_n, \Psi_n))}$$
$$\le\sqrt{\rho(S_{\alpha_1}(\Psi_1,  \Sigma_1)\cdots S_{\alpha_1}(\Psi_n,  \Sigma_n))^{\alpha}\rho(S_{\alpha_1}( \Sigma_1, \Psi_1)\cdots S_{\alpha_1}( \Sigma_n, \Psi_n))^{1-\alpha}}\times$$
$$\sqrt{\rho(S_{\alpha_1}( \Sigma_1, \Psi_1)\cdots S_{\alpha_1}( \Sigma_n, \Psi_n))^{\alpha}\rho(S_{\alpha_1}(\Psi_1,  \Sigma_1)\cdots S_{\alpha_1}(\Psi_n,  \Sigma_n))^{1-\alpha}}=\tilde{\rho}_{n}(\alpha_1),$$
which proves that $\tilde{\rho}_{n}$ is decreasing on $[0, \frac{1}{2}]$. 

Similarly, in the case $\frac{1}{2}\le\alpha_1<\alpha_2\le 1$ let $\alpha=\frac{\alpha_1+\alpha_2-1}{2\alpha_2-1} \in (0, 1)$. Then 
\be
S_{\alpha_1}(\Psi, \Sigma) \subset S_{\alpha}(S_{\alpha_2}(\Psi, \Sigma), S_{\alpha_2}(\Sigma, \Psi)).
\label{good2brezset}
\ee 
Indeed, for each $K\in \Psi$ and $H \in \Sigma$ we have
$$S_{\alpha_1}(K, H)=K^{(\alpha_1)}\circ H^{(1-\alpha_1)}=K^{\alpha_2\alpha+(1-\alpha_2)(1-\alpha)}\circ H^{\alpha(1-\alpha_2)+\alpha_2(1-\alpha)}$$
$$=(K^{(\alpha_2)}\circ H^{(1-\alpha_2)})^{(\alpha)}\circ (H^{(\alpha_2)}\circ K^{(1-\alpha_2)})^{(1-\alpha)}$$
\be
=S_{\alpha}(S_{\alpha_2}(K, H), S_{\alpha_2}(H, K)) \subset S_{\alpha}(S_{\alpha_2}(\Psi, \Sigma), S_{\alpha_2}(\Sigma, \Psi))
\label{good2brez}
\ee
which proves (\ref{good2brezset}).
By (\ref{good2brezset}) and (\ref{aufstand2dvabrezsets}) we obtain
$$\tilde{\rho}_{n}(\alpha_1)=\sqrt{\rho(S_{\alpha_1}(\Psi_1, \Sigma_1)\cdots S_{\alpha_1}(\Psi_n, \Sigma_n))
\rho(S_{\alpha_1}(\Sigma_1, \Psi_1)\cdots S_{\alpha_1}(\Sigma_n, \Psi_n))}$$
$$\le\sqrt{\rho(S_{\alpha_2}(\Psi_1, \Sigma_1)\cdots S_{\alpha_2}(\Psi_n, \Sigma_n))^{\alpha}\rho(S_{\alpha_2}(\Sigma_1, \Psi_1)\cdots S_{\alpha_2}(\Sigma_n, \Psi_n))^{1-\alpha}}\times$$
$$\sqrt{\rho(S_{\alpha_2}(\Sigma_1, \Psi_1)\cdots S_{\alpha_2}(\Sigma_n, \Psi_n))^{\alpha}\rho(S_{\alpha_2}(\Psi_1, \Sigma_1)\cdots S_{\alpha_2}(\Psi_n, \Sigma_n))^{1-\alpha}}=\tilde{\rho}_{n}(\alpha_2),$$
which proves that $\tilde{\rho}_{n}$ is increasing on $[\frac{1}{2},1]$.

The properties for $\rho \in \{\hat{r}_{ess}, r_{ess}, \gamma \}$ and $\rho =w$ under the assumptions of the theorem are proved similarly.
\end{proof}

The following corollaries extend \cite[Corollaries 9 and 10]{BP25}.
\begin{corollary} Let $\Psi$ and $\Sigma$ be bounded sets of positive kernel operators on a Banach function space $L$ 
 and let $\rho \in \{ r, \hat{r},  \|\cdot\|\}$. Then a function 
$\tilde{\rho} _{1}(\alpha) = \sqrt{\rho(S_{\alpha}(\Psi, \Sigma)) \rho(S_{\alpha}(\Sigma, \Psi))} $ is  decreasing on $[0, \frac{1}{2}]$ and increasing on $[\frac{1}{2}, 1]$. In particular, 
$$\sqrt{\rho(S_{\alpha}(\Psi, \Sigma)) \rho(S_{\alpha}(\Sigma, \Psi))} \ge \rho(S_{\frac{1}{2}}(\Psi, \Sigma)) = \rho (\Psi^{\left(\frac{1}{2}\right)} \circ \Sigma^{\left(\frac{1}{2}\right)})$$
for all $\alpha \in [0,1]$.

If, in addition $L$ and $L^*$ have order continuous norms, then the same conclusions hold also for $\rho \in \{r_{ess}, \hat{r}_{ess}, \gamma \}$ and if, in addition, $L=L^2(X, \mu)$ then  the same conclusions hold also for $\rho = w$.

\end{corollary}
\begin{corollary} Let $\Psi_1, \Psi_2$ and $\Sigma_1, \Sigma_2$ be bounded sets of positive kernel operators on a Banach function space $L$ and let $\rho \in \{r, \hat{r},\|\cdot\| \}$. Then a function 
$$\tilde{\rho} _{2}(\alpha)=\sqrt{\rho(S_{\alpha}(\Psi_1, \Sigma_1)S_{\alpha}(\Psi_2, \Sigma_2))\rho(S_{\alpha}(\Sigma_1, \Psi_1)S_{\alpha}(\Sigma_2, \Psi_2))}$$
is  decreasing on $[0, \frac{1}{2}]$ and increasing on $[\frac{1}{2}, 1]$. In particular, 
$$\tilde{\rho} _{2}(\alpha) \ge  \tilde{\rho} _{2}\left(\frac{1}{2}\right)= \rho \left ((\Psi^{\left(\frac{1}{2}\right)} _1 \circ \Sigma^{\left(\frac{1}{2}\right)} _1) (\Psi^{\left(\frac{1}{2}\right)} _2  \circ \Sigma^{\left(\frac{1}{2}\right)} _2)\right) $$ 
for all $\alpha \in [0,1]$.

If, in addition $L$ and $L^*$ have order continuous norms, then the same conclusions hold also for $\rho \in \{r_{ess}, \hat{r}_{ess}, \gamma \}$ and if, in addition, $L=L^2(X, \mu)$ then  the same conclusions hold also for $\rho = w$.
\label{dvojka}
\end{corollary}

%
The following results  for $T_{\alpha}$ follow from Theorem  \ref{S_alpha_simplifiedsets}. 
\begin{corollary}
For bounded sets $\Psi_1, \ldots , \Psi_n, \Sigma_1, \ldots , \Sigma_n$ of positive kernel operators  on $L=L^2(X, \mu)$, $\alpha \in [0,1]$ and for $\rho \in \{ r, \hat{r},  r_{ess}, \hat{r}_{ess}, \gamma, \|\cdot\|, w \}$ 
define 
$$\tilde{\rho} ^* _n (\alpha):=\sqrt{\rho(T_{\alpha}(\Psi_1, \Sigma_1)\cdots T_{\alpha}(\Psi_n, \Sigma_n))\rho(T_{\alpha}(\Sigma_n, \Psi_n)\cdots T_{\alpha}(\Sigma_1, \Psi_1))}.$$
 Then $\tilde{\rho}^* _{n}$ is decreasing on $[0, \frac{1}{2}]$ and increasing on $[\frac{1}{2}, 1]$. In particular,  \\
 $\tilde{\rho} ^* _{n}(\alpha) \ge \tilde{\rho} ^* _{n} (\frac{1}{2})$ for all $\alpha \in [0,1]$.
\label{cor_simpler_sets}
\end{corollary}
\begin{proof}
First observe that by definitions and the properties of the adjoints we have
$$\tilde{\rho} ^* _n (\alpha)=\sqrt{\rho(T_{\alpha}(\Psi_1, \Sigma_1)\cdots T_{\alpha}(\Psi_n, \Sigma_n))
\rho(T_{\alpha}(\Sigma_n, \Psi_n)\cdots T_{\alpha}(\Sigma_1, \Psi_1))}$$
$$=\sqrt{\rho(S_{\alpha}(\Psi_1, \Sigma_1 ^*)\cdots S_{\alpha}(\Psi_n, \Sigma_n ^*))
\rho(S_{\alpha}(\Sigma_n, \Psi_n ^*)\cdots S_{\alpha}(\Sigma_1 , \Psi_1 ^*))}$$
$$=\sqrt{\rho(S_{\alpha}(\Psi_1, \Sigma_1 ^*)\cdots S_{\alpha}(\Psi_n, \Sigma_n ^*))\rho(S_{\alpha}( \Sigma_1 ^* , \Psi_1  )\cdots S_{\alpha}(\Sigma_n ^*  , \Psi_n ))}.$$ 
Thus $\tilde{\rho} ^* _n$  is a map from Theorem \ref{S_alpha_simplifiedsets} associated to sets $\Psi_1, \ldots , \Psi_n, \Sigma_1 ^*, \ldots ,  \Sigma_n ^*$ and the result follows.
\end{proof}

\begin{corollary}
For  bounded sets $\Psi_1, \ldots , \Psi_n, \Sigma_1, \ldots , \Sigma_n$ of positive kernel operators  on $L=L^2(X, \mu)$, $\alpha \in [0,1]$ and for $\rho \in \{ r, \hat{r},  r_{ess}, \hat{r}_{ess}, \gamma, \|\cdot\|, w \}$ 
define 
$$\rho ^* _n (\alpha):=\sqrt[4]{\rho(T_{\alpha}(\Psi_1, \Sigma_1)\cdots T_{\alpha}(\Psi_n,  \Sigma_n))\rho(T_{\alpha}( \Sigma_n, \Psi_n)\cdots T_{\alpha}( \Sigma_1, \Psi_1))}\times$$
$$\sqrt[4]{\rho(T_{\alpha}(\Psi_n,  \Sigma_n)\cdots T_{\alpha}(\Psi_1,  \Sigma_1))\rho(T_{\alpha}( \Sigma_1, \Psi_1)\cdots T_{\alpha}( \Sigma_n, \Psi_n))}.$$
 Then $\rho^* _{n}$ is decreasing on $[0, \frac{1}{2}]$ and increasing on $[\frac{1}{2}, 1]$. In particular,  \\
 $\rho ^* _{n}(\alpha) \ge \rho ^* _{n} (\frac{1}{2})$ for all $\alpha \in [0,1]$.
\end{corollary}
\begin{proof}
The map $\rho ^* _n$ is a geometric mean of a map $\tilde{\rho} ^* _n $ from Corollary \ref{cor_simpler_sets} and another such map where the sets $\Psi_1, \ldots , \Psi_n$ are replaced by $\Psi_n, \ldots , \Psi_1$, respectively, and  the sets $ \Sigma_1, \ldots ,  \Sigma_n$ are replaced by $ \Sigma_n, \ldots ,  \Sigma_1$, respectively. Thus the result follows from  Corollary \ref{cor_simpler_sets}.
\end{proof}
Even the following five results, which are special cases for singleton sets of  the above results, are new.  

\begin{corollary}
\label{S_alpha_simplified}
For positive kernel operators $K_1, \ldots , K_n, H_1, \ldots , H_n$ on  a Banach function space $L$, 
$\alpha \in [0,1]$ and for $\rho \in \{ r, \|\cdot\|\}$ 
define 
$$\tilde{\rho} _{n}(\alpha):=\sqrt{\rho(S_{\alpha}(K_1, H_1)\cdots S_{\alpha}(K_n, H_n))
\rho(S_{\alpha}(H_1, K_1)\cdots S_{\alpha}(H_n, K_n))}.$$
 Then $\tilde{\rho}_{n}$ is decreasing on $[0, \frac{1}{2}]$ and increasing on $[\frac{1}{2}, 1]$. In particular,  $\tilde{\rho} _{n}(\alpha) \ge \tilde{\rho} _{n}(\frac{1}{2})$ for all $\alpha \in [0,1]$.
 
 If, in addition $L$ and $L^*$ have order continuous norms, then the same conclusions hold also for $\rho \in \{r_{ess}, \gamma \}$ and if, in addition, $L=L^2(X, \mu)$ then  the same conclusions hold also for $\rho = w$.
\end{corollary}

\begin{corollary} Let $K$ and $H$ be positive kernel operators on a Banach function space $L$ 
and let $\rho \in \{ r, \|\cdot\| \}$. Then a function 
$\tilde{\rho} _{1}(\alpha) = \sqrt{\rho(S_{\alpha}(K, H)) \rho(S_{\alpha}(H, K))} $ is  decreasing on $[0, \frac{1}{2}]$ and increasing on $[\frac{1}{2}, 1]$. In particular, 
$$\sqrt{\rho(S_{\alpha}(K, H)) \rho(S_{\alpha}(H, K))} \ge \rho(S_{\frac{1}{2}}(K, H)) = \rho (K^{\left(\frac{1}{2}\right)} \circ H^{\left(\frac{1}{2}\right)})$$
for all $\alpha \in [0,1]$.

 If, in addition $L$ and $L^*$ have order continuous norms, then the same conclusions hold also for $\rho \in \{r_{ess}, \gamma \}$ and if, in addition, $L=L^2(X, \mu)$ then  the same conclusions hold also for $\rho = w$.
\end{corollary}
\begin{corollary} Let $K_1, K_2$ and $H_1, H_2$ be positive kernel operators on a Banach function space  $L$ 
and let $\rho \in \{ r, \|\cdot\|\}$. Then a function 
$$\tilde{\rho} _{2}(\alpha)=\sqrt{\rho(S_{\alpha}(K_1, H_1)S_{\alpha}(K_2, H_2))\rho(S_{\alpha}(H_1, K_1)S_{\alpha}(H_2, K_2))}$$
is  decreasing on $[0, \frac{1}{2}]$ and increasing on $[\frac{1}{2}, 1]$. In particular, 
$$\tilde{\rho} _{2}(\alpha) \ge  \tilde{\rho} _{2}\left(\frac{1}{2}\right)= \rho \left ((K^{\left(\frac{1}{2}\right)} _1 \circ H^{\left(\frac{1}{2}\right)} _1) (K^{\left(\frac{1}{2}\right)} _2  \circ H^{\left(\frac{1}{2}\right)} _2)\right) $$ 
for all $\alpha \in [0,1]$.

 If, in addition $L$ and $L^*$ have order continuous norms, then the same conclusions hold also for $\rho \in \{r_{ess}, \gamma \}$ and if, in addition, $L=L^2(X, \mu)$ then  the same conclusions hold also for $\rho = w$.
\label{dvojka}
\end{corollary}

\begin{corollary}
For positive kernel operators $K_1, \ldots , K_n, H_1, \ldots , H_n$ on $L=L^2(X, \mu)$, $\alpha \in [0,1]$ and for $\rho \in \{ r, r_{ess}, \gamma, \|\cdot\|, w \}$ 
define 
$$\tilde{\rho} ^* _n (\alpha):=\sqrt{\rho(T_{\alpha}(K_1, H_1)\cdots T_{\alpha}(K_n, H_n))\rho(T_{\alpha}(H_n, K_n)\cdots T_{\alpha}(H_1, K_1))}.$$
 Then $\tilde{\rho}^* _{n}$ is decreasing on $[0, \frac{1}{2}]$ and increasing on $[\frac{1}{2}, 1]$. In particular,  \\
 $\tilde{\rho} ^* _{n}(\alpha) \ge \tilde{\rho} ^* _{n} (\frac{1}{2})$ for all $\alpha \in [0,1]$.
\label{cor_simpler}
\end{corollary}

\begin{corollary}
\label{naL2}
For positive kernel operators $K_1, \ldots , K_n, H_1, \ldots , H_n$ on $L=L^2(X, \mu)$, $\alpha \in [0,1]$ and for $\rho \in \{ r, r_{ess}, \gamma, \|\cdot\|, w \}$ 
define 
$$\rho ^* _n (\alpha):=\sqrt[4]{\rho(T_{\alpha}(K_1, H_1)\cdots T_{\alpha}(K_n, H_n))\rho(T_{\alpha}(H_n, K_n)\cdots T_{\alpha}(H_1, K_1))}\times$$
$$\sqrt[4]{\rho(T_{\alpha}(K_n, H_n)\cdots T_{\alpha}(K_1, H_1))\rho(T_{\alpha}(H_1, K_1)\cdots T_{\alpha}(H_n, K_n))}.$$
 Then $\rho^* _{n}$ is decreasing on $[0, \frac{1}{2}]$ and increasing on $[\frac{1}{2}, 1]$. In particular,  \\
 $\rho ^* _{n}(\alpha) \ge \rho ^* _{n} (\frac{1}{2})$ for all $\alpha \in [0,1]$.
\end{corollary}

\begin{remark}{\rm We would like to acknowledge that observing that Corollary \ref{naL2} is valid was the starting point of this research.
}
\end{remark}

The following result was already proved in \cite[Lemma 12]{BP25}.

\begin{lemma} Let $\Psi_{ij}$ and $\Sigma_{ij}$ for $i=1, \ldots , n$ and $j=1, \ldots , m$ be bounded sets of positive kernel operators on a Banach function space $L$ and let $\alpha\in [0, 1]$. If 
$\rho \in \{r, \hat{r}, \|\cdot\|\}$, then
$$\rho\left( \left((\Psi^{(\alpha)} _{11} \circ \Sigma^{(1-\alpha)} _{11} ) + \cdots + (\Psi^{(\alpha)} _{1m} \circ \Sigma^{(1-\alpha)} _{1m} )\right)  \cdots  \left((\Psi^{(\alpha)} _{n1} \circ \Sigma^{(1-\alpha)} _{n1} ) + \cdots + (\Psi^{(\alpha)} _{nm} \circ \Sigma^{(1-\alpha)} _{nm} )\right)\right)$$
$$\le\rho ( \left((\Psi_{11}+\cdots +\Psi_{1m})^{(\alpha)} \circ  (\Sigma_{11}+\cdots +\Sigma_{1m})^{(1-\alpha)}\right) \cdots $$
\be
\cdots \left((\Psi_{n1}+\cdots +\Psi_{nm})^{(\alpha)} \circ  (\Sigma_{n1}+\cdots +\Sigma_{nm})^{(1-\alpha)}\right) ).
\label{expo_gen1}
\ee

If, in  addition, $L$ and $L^*$ have order continuous norms, then (\ref{expo_gen1}) holds also for   $\rho \in  \{r_{ess}, \hat{r}_{ess}, \gamma \}$.

If, in  addition, $L=L^2(X, \mu)$, then (\ref{expo_gen1}) holds also for   $\rho =w$.

\end{lemma}
Observe that (\ref{expo_gen1}) with the notations used in this article reads as
$$\rho\left( \left(S_{\alpha}(\Psi_{11}, \Sigma_{11} ) + \cdots + S_{\alpha} (\Psi_{1m},\Sigma_{1m} )\right)  \cdots  \left(S_{\alpha}(\Psi_{n1},\Sigma_{n1} ) + \cdots + 
S_{\alpha}(\Psi_{nm}, \Sigma_{nm} )\right)\right)$$
\be
\le\rho ( S_{\alpha}(\Psi_{11}+\cdots +\Psi_{1m}, \Sigma_{11}+\cdots +\Sigma_{1m})\cdots S_{\alpha}(\Psi_{n1}+\cdots +\Psi_{nm} , \Sigma_{n1}+\cdots +\Sigma_{nm}) ).
\label{expo_gen1notation}
\ee
The following theorem generalizes Theorem \ref{th2_joint}. 
\begin{theorem}
\label{th2_jointgen}
Let $\Psi_{ij}$ and $\Sigma _{ij}$ for $i=1, \ldots , n$ and $j=1, \ldots , m$ be bounded sets of positive kernel operators on a Banach function space $L$ 
and $\rho \in \{r, \hat{r}, \|\cdot\|\} $. 
Define $\overline{\rho} _n: [0,1] \to [0, \infty)$ by \\
$$\overline{\rho}_n (\alpha)=\rho ^{\frac{1}{2}} \left( \left(S_{\alpha}(\Psi_{11}, \Sigma_{11} ) + \cdots + S_{\alpha} (\Psi_{1m},\Sigma_{1m} )\right)  \cdots  \left(S_{\alpha}(\Psi_{n1},\Sigma_{n1} ) + \cdots + 
S_{\alpha}(\Psi_{nm}, \Sigma_{nm} )\right)\right)\times $$
$$\rho ^{\frac{1}{2}} \left( \left(S_{\alpha}( \Sigma_{11}, \Psi_{11} ) + \cdots + S_{\alpha} ( \Sigma_{1m},\Psi_{1m} )\right)  \cdots  \left(S_{\alpha}( \Sigma_{n1},\Psi_{n1} ) + \cdots + 
S_{\alpha}( \Sigma_{nm}, \Psi_{nm} )\right)\right).$$
 Then
\be
\overline{\rho} _n (\alpha) \le \rho (\left (\Psi_{11} + \cdots + \Psi_{1m})\cdots (\Psi_{n1} + \cdots + \Psi_{nm}) \right)^{\frac{1}{2}} \times
\label{cutegen} 
\ee
$$\rho((\Sigma_{11}+\cdots +\Sigma_{1m})\cdots(\Sigma_{n1}+\cdots +\Sigma_{nm})^{\frac{1}{2}}$$
for each $\alpha \in [0,1]$. Moreover, $\overline{\rho}_n$ is decreasing on $[0,\frac{1}{2}]$ and increasing on $[\frac{1}{2}, 1]$.

 If, in addition $L$ and $L^*$ have order continuous norms, then the same conclusions hold also for $\rho \in \{r_{ess}, \hat{r}_{ess}, \gamma \}$ and if, in addition, $L=L^2(X, \mu)$ then  the same conclusions hold also for $\rho = w$.
\end{theorem}
\begin{proof}
Let $\rho \in \{r, \hat{r}, \|\cdot\|\} $.
By double application of (\ref{expo_gen1notation}) and then of (\ref{aufstand2dvabrezsets}) we obtain
$$\overline{\rho}_n (\alpha)\le\rho ( S_{\alpha}(\Psi_{11}+\cdots +\Psi_{1m}, \Sigma_{11}+\cdots +\Sigma_{1m})\cdots S_{\alpha}(\Psi_{n1}+\cdots +\Psi_{nm} , \Sigma_{n1}+\cdots +\Sigma_{nm}) )^{1/2}\times$$
$$\rho ( S_{\alpha}(\Sigma_{11}+\cdots +\Sigma_{1m}, \Psi_{11}+\cdots +\Psi_{1m})\cdots S_{\alpha}(\Sigma_{n1}+\cdots +\Sigma_{nm} , \Psi_{n1}+\cdots +\Psi_{nm}) )^{1/2}$$
$$\le\rho((\Psi_{11}+\cdots +\Psi_{1m})\cdots(\Psi_{n1}+\cdots +\Psi_{nm}))^{\frac{\alpha}{2}}\times$$
$$\rho((\Sigma_{11}+\cdots +\Sigma_{1m})\cdots(\Sigma_{n1}+\cdots +\Sigma_{nm}))^{\frac{1-\alpha}{2}}\times$$
$$\rho((\Sigma_{11}+\cdots +\Sigma_{1m})\cdots(\Sigma_{n1}+\cdots +\Sigma_{nm}))^{\frac{\alpha}{2}}\times$$
$$\rho((\Psi_{11}+\cdots +\Psi_{1m})\cdots(\Psi_{n1}+\cdots +\Psi_{nm}))^{\frac{1-\alpha}{2}}$$
$$=\rho((\Psi_{11}+\cdots +\Psi_{1m})\cdots(\Psi_{n1}+\cdots +\Psi_{nm}))^{\frac{1}{2}}\times$$
$$\rho((\Sigma_{11}+\cdots +\Sigma_{1m})\cdots(\Sigma_{n1}+\cdots +\Sigma_{nm}))^{\frac{1}{2}}$$
which proves (\ref{cutegen}).
Let $0\le\alpha_1<\alpha_2\le \frac{1}{2}$ and let $\Psi$ and $\Sigma$ be bounded sets of positive kernel operators on $L$. For $\alpha=\frac{\alpha_1+\alpha_2-1}{2\alpha_1-1}$ we have $\alpha\in (0, 1)$ and $S_{\alpha_2}(\Psi, \Sigma) \subset S_{\alpha}(S_{\alpha_1}(\Psi, \Sigma), S_{\alpha_1}(\Sigma, \Psi))$ by (\ref{good1brezset}). 
Then by (\ref{cutegen})
$$\overline{\rho}_n (\alpha _2)\le$$
$$\left(\rho((S_{\alpha}( S_{\alpha_1}(\Psi_{11}, \Sigma_{11} ), S_{\alpha_1}(\Sigma_{11}, \Psi_{11}) )  + \cdots + S_{\alpha}( S_{\alpha_1}(\Psi_{1m}, \Sigma_{1m} ),  S_{\alpha_1}(\Sigma_{1m}, \Psi_{1m}) )\right)  \cdots $$
$$\cdots(S_{\alpha}( S_{\alpha_1}(\Psi_{n1}, \Sigma_{n1} ), S_{\alpha_1}(\Sigma_{n1}, \Psi_{n1}) )  + \cdots + S_{\alpha}( S_{\alpha_1}(\Psi_{nm}, \Sigma_{nm} ),  S_{\alpha_1}(\Sigma_{nm}, \Psi_{nm}))^{\frac{1}{2}}\times $$
$$\left(\rho((S_{\alpha}( S_{\alpha_1}(\Sigma_{11}, \Psi_{11} ), S_{\alpha_1}(\Psi_{11}, \Sigma_{11}) )  + \cdots + S_{\alpha}( S_{\alpha_1}(\Sigma_{1m}, \Psi_{1m} ),  S_{\alpha_1}(\Psi_{1m}, \Sigma_{1m}) )\right)  \cdots $$
$$\cdots(S_{\alpha}( S_{\alpha_1}(\Sigma_{n1}, \Psi_{n1} ), S_{\alpha_1}(\Psi_{n1}, \Sigma_{n1}) )  + \cdots + S_{\alpha}( S_{\alpha_1}(\Sigma_{nm}, \Psi_{nm} ),  S_{\alpha_1}(\Psi_{nm}, \Sigma_{nm}))^{\frac{1}{2}}$$ 
$$\le \left(\rho((S_{\alpha_1} (\Psi_{11}, \Sigma_{11}) + \cdots + S_{\alpha_1}(\Psi_{1m}, \Sigma_{1m}))  \cdots (S_{\alpha_1} (\Psi_{n1}, \Sigma_{n1}) + \cdots + S_{\alpha_1}(\Psi_{nm}, \Sigma_{nm})))\right)^{\frac{1}{4}}\times $$
$$  \left(\rho((S_{\alpha_1} (\Sigma_{11}, \Psi_{11}) + \cdots + S_{\alpha_1}(\Sigma_{1m}, \Psi_{1m}))  \cdots (S_{\alpha_1} (\Sigma_{n1}, \Psi_{n1}) + \cdots + S_{\alpha_1}(\Sigma_{nm}, \Psi_{nm})))\right)^{\frac{1}{4}}\times $$
$$ \left(\rho((S_{\alpha_1} (\Sigma_{11}, \Psi_{11}) + \cdots + S_{\alpha_1}(\Sigma_{1m}, \Psi_{1m}))  \cdots (S_{\alpha_1} (\Sigma_{n1}, \Psi_{n1}) + \cdots + S_{\alpha_1}(\Sigma_{nm}, \Psi_{nm})))\right)^{\frac{1}{4}}\times $$
$$ \left(\rho((S_{\alpha_1} (\Psi_{11}, \Sigma_{11}) + \cdots + S_{\alpha_1}(\Psi_{1m}, \Sigma_{1m}))  \cdots (S_{\alpha_1} (\Psi_{n1}, \Sigma_{n1}) + \cdots + S_{\alpha_1}(\Psi_{nm}, \Sigma_{nm})))\right)^{\frac{1}{4}}$$
$$=\overline{\rho}_n (\alpha _1),$$
which proves that  $\overline{\rho}_n$ is decreasing on $[0, \frac{1}{2}]$.

To prove that $\overline{\rho}_n$ is increasing on $[\frac{1}{2}, 1]$ let $\frac{1}{2}\le\alpha_1<\alpha_2\le 1$. For $\alpha=\frac{\alpha_1+\alpha_2-1}{2\alpha_2-1}$ we have $\alpha\in (0, 1)$ and $S_{\alpha_1}(\Psi, \Sigma) \subset S_{\alpha}(S_{\alpha_2}(\Psi, \Sigma), S_{\alpha_2}(\Sigma, \Psi))$ by (\ref{good2brezset}). Similarly as above it follows from  (\ref{cutegen}) that $\overline{\rho}_n (\alpha _1) \le \overline{\rho}_n (\alpha _2)  $, which completes the proof in the case when $\rho \in \{r, \hat{r}, \|\cdot\|\}$.

The properties for $\rho \in \{\hat{r}_{ess}, r_{ess}, \gamma \}$ and $\rho =w$ under the assumptions of the theorem are proved in a similar manner.
\end{proof}

\begin{corollary}
\label{th2_jointgen1}
Let $\Psi_{j}$ and $\Sigma _{j}$ for $j=1, \ldots , m$ be bounded sets of positive kernel operators on a Banach function space $L$ 
and $\rho \in \{r, \hat{r}, \|\cdot\|\} $. 
For $\overline{\rho} _1: [0,1] \to [0, \infty)$, defined by \\
$$\overline{\rho}_1 (\alpha)=\rho ^{\frac{1}{2}}  \left(S_{\alpha}(\Psi_{1}, \Sigma_{1} ) + \cdots + S_{\alpha} (\Psi_{m},\Sigma_{m})  \right)
\rho ^{\frac{1}{2}}  \left(S_{\alpha}( \Sigma_{1}, \Psi_{1} ) + \cdots + S_{\alpha} ( \Sigma_{m},\Psi_{m})  \right),$$
 we have
\be
\overline{\rho} _1 (\alpha) \le \rho  (\Psi_{1} + \cdots + \Psi_{m})^{\frac{1}{2}} \rho(\Sigma_{1}+\cdots +\Sigma_{m})^{\frac{1}{2}}
\label{cutegen1} 
\ee
for each $\alpha \in [0,1]$. Moreover, $\overline{\rho}_1$ is decreasing on $[0,\frac{1}{2}]$ and increasing on $[\frac{1}{2}, 1]$.

 If, in addition $L$ and $L^*$ have order continuous norms, then the same conclusions hold also for $\rho \in \{r_{ess}, \hat{r}_{ess}, \gamma \}$ and if, in addition, $L=L^2(X, \mu)$ then  the same conclusions hold also for $\rho = w$.
\end{corollary}

\begin{corollary}
\label{th2_jointgen2}
Let $\Psi_{ij}$ and $\Sigma _{ij}$ for $i=1, 2$ and $j=1, \ldots , m$ be bounded sets of positive kernel operators on a Banach function space $L$ 
and $\rho \in \{r, \hat{r}, \|\cdot\|\} $. 
For $\overline{\rho} _2: [0,1] \to [0, \infty)$, defined by, \\
$$\overline{\rho}_2 (\alpha)=\rho ^{\frac{1}{2}} \left( \left(S_{\alpha}(\Psi_{11}, \Sigma_{11} ) + \cdots + S_{\alpha} (\Psi_{1m},\Sigma_{1m} )\right)\left(S_{\alpha}(\Psi_{21},\Sigma_{21} ) + \cdots + 
S_{\alpha}(\Psi_{2m}, \Sigma_{2m} )\right)\right)\times $$
$$\rho ^{\frac{1}{2}} \left( \left(S_{\alpha}( \Sigma_{11}, \Psi_{11} ) + \cdots + S_{\alpha} ( \Sigma_{1m},\Psi_{1m} )\right)\left(S_{\alpha}( \Sigma_{21},\Psi_{21} ) + \cdots + 
S_{\alpha}( \Sigma_{2m}, \Psi_{2m} )\right)\right),$$
we have
\be
\overline{\rho} _2 (\alpha) \le \rho (\left (\Psi_{11} + \cdots + \Psi_{1m})(\Psi_{21} + \cdots + \Psi_{2m}) \right)^{\frac{1}{2}} \times
\label{cutegen2} 
\ee
$$\rho((\Sigma_{11}+\cdots +\Sigma_{1m})(\Sigma_{21}+\cdots +\Sigma_{2m}))^{\frac{1}{2}}$$
for each $\alpha \in [0,1]$. Moreover, $\overline{\rho}_2$ is decreasing on $[0,\frac{1}{2}]$ and increasing on $[\frac{1}{2}, 1]$.

 If, in addition $L$ and $L^*$ have order continuous norms, then the same conclusions hold also for $\rho \in \{r_{ess}, \hat{r}_{ess}, \gamma \}$ and if, in addition, $L=L^2(X, \mu)$ then  the same conclusions hold also for $\rho = w$.
\end{corollary}
\begin{corollary}
Let $\Psi_{ij}$ and $\Sigma _{ij}$ for $i=1, \ldots , n$ and $j=1, \ldots , m$ be
 bounded sets  of positive kernel operators  on $L=L^2(X, \mu)$, $\alpha \in [0,1]$ and let \\
  $\rho \in \{ r, \hat{r},  r_{ess}, \hat{r}_{ess}, \gamma, \|\cdot\|, w \}$. 
Define $\overline{\rho}^* _n: [0,1] \to [0, \infty)$ by \\
$$\overline{\rho}^* _n (\alpha)=\rho ^{\frac{1}{2}} \left( \left(T_{\alpha}(\Psi_{11}, \Sigma_{11} ) + \cdots + T_{\alpha} (\Psi_{1m},\Sigma_{1m} )\right)  \cdots  \left(T_{\alpha}(\Psi_{n1},\Sigma_{n1} ) + \cdots + 
T_{\alpha}(\Psi_{nm}, \Sigma_{nm} )\right)\right)\times $$
$$\rho ^{\frac{1}{2}} \left(\left(T_{\alpha}( \Sigma_{n1},\Psi_{n1} ) + \cdots + 
T_{\alpha}( \Sigma_{nm}, \Psi_{nm} )\right) \cdots \left(T_{\alpha}( \Sigma_{11}, \Psi_{11} ) + \cdots + T_{\alpha} ( \Sigma_{1m},\Psi_{1m} )\right)  \right).$$
 Then
$\overline{\rho}^* _n$ is decreasing on $[0,\frac{1}{2}]$ and increasing on $[\frac{1}{2}, 1]$. In particular,  
 $\overline{\rho}^* _n (\alpha) \ge\overline{\rho}^* _n (\frac{1}{2})$ for all $\alpha \in [0,1]$.

\label{cor_simpler_sets_sums}
\end{corollary}
\begin{proof}
First observe that by definitions and the properties of the adjoints we have
$$\overline{\rho}^* _n (\alpha)=$$
$$\rho ^{\frac{1}{2}} \left( \left(S_{\alpha}(\Psi_{11}, \Sigma_{11} ^* ) + \cdots + S_{\alpha} (\Psi_{1m},\Sigma_{1m} ^* )\right)  \cdots  \left(S_{\alpha}(\Psi_{n1},\Sigma_{n1} ^* ) + \cdots + 
S_{\alpha}(\Psi_{nm}, \Sigma_{nm} ^* )\right)\right)\times $$
$$\rho ^{\frac{1}{2}} \left(\left(S_{\alpha}( \Sigma_{n1},\Psi_{n1} ^* ) + \cdots + 
S_{\alpha}( \Sigma_{nm}, \Psi_{nm} ^* )\right) \cdots \left(S_{\alpha}( \Sigma_{11}, \Psi_{11} ^* ) + \cdots + S_{\alpha} ( \Sigma_{1m},\Psi_{1m} ^* )\right)  \right)$$
$$=\rho ^{\frac{1}{2}} \left( \left(S_{\alpha}(\Psi_{11}, \Sigma_{11} ^* ) + \cdots + S_{\alpha} (\Psi_{1m},\Sigma_{1m} ^* )\right)  \cdots  \left(S_{\alpha}(\Psi_{n1},\Sigma_{n1} ^* ) + \cdots + 
S_{\alpha}(\Psi_{nm}, \Sigma_{nm} ^* )\right)\right)\times $$
$$\rho ^{\frac{1}{2}} \left( \left(S_{\alpha}( \Sigma_{11}^*, \Psi_{11}  ) + \cdots + S_{\alpha} ( \Sigma_{1m} ^*,\Psi_{1m} )\right)  \right)\cdots \left(S_{\alpha}( \Sigma_{n1} ^*,\Psi_{n1}  ) + \cdots + 
S_{\alpha}( \Sigma_{nm} ^*, \Psi_{nm} )\right) $$
Thus $\overline{\rho}^* _n $  is a map from Theorem \ref{th2_jointgen} associated to sets $\Psi_{ij}$ and $\Sigma ^* _{ij}$ for $i=1, \ldots , n$ and $j=1, \ldots , m$ and the result follows.
\end{proof}

We conclude the article by two corollaries for the special case of singleton sets.
\begin{corollary}
\label{th2_jointgen3}
Let $K_{ij}$ and $H_{ij}$ for $i=1, \ldots , n$ and $j=1, \ldots , m$ be positive kernel operators on a Banach function space $L$ 
and $\rho \in \{r, \|\cdot\|\} $. 
For $\overline{\rho} _n: [0,1] \to [0, \infty)$, defined by \\
$$\overline{\rho}_n (\alpha)=\rho ^{\frac{1}{2}} \left( \left(S_{\alpha}(K_{11}, H_{11} ) + \cdots + S_{\alpha} (K_{1m},H_{1m} )\right)  \cdots  \left(S_{\alpha}(K_{n1},H_{n1} ) + \cdots + 
S_{\alpha}(K_{nm}, H_{nm} )\right)\right)\times $$
$$\rho ^{\frac{1}{2}} \left( \left(S_{\alpha}( H_{11}, K_{11} ) + \cdots + S_{\alpha} ( H_{1m},K_{1m} )\right)  \cdots  \left(S_{\alpha}( H_{n1},K_{n1} ) + \cdots + 
S_{\alpha}( H_{nm}, K_{nm} )\right)\right),$$
 we have
\be
\overline{\rho} _n (\alpha) \le \rho (\left (K_{11} + \cdots + K_{1m})\cdots (K_{n1} + \cdots + K_{nm}) \right)^{\frac{1}{2}} \times
\label{cutegen3} 
\ee
$$\rho((H_{11}+\cdots +H_{1m})\cdots(H_{n1}+\cdots +H_{nm})^{\frac{1}{2}}$$
for each $\alpha \in [0,1]$. Moreover, $\overline{\rho}_n$ is decreasing on $[0,\frac{1}{2}]$ and increasing on $[\frac{1}{2}, 1]$.

 If, in addition $L$ and $L^*$ have order continuous norms, then the same conclusions hold also for $\rho \in \{r_{ess}, \gamma \}$ and if, in addition, $L=L^2(X, \mu)$ then  the same conclusions hold also for $\rho = w$.
\end{corollary}

\begin{corollary}
Let $K_{ij}$ and $H_{ij}$ for $i=1, \ldots , n$ and $j=1, \ldots , m$ be
 bounded sets  of positive kernel operators  on $L=L^2(X, \mu)$, $\alpha \in [0,1]$ and let \\
  $\rho \in \{ r, \hat{r},  r_{ess}, \hat{r}_{ess}, \gamma, \|\cdot\|, w \}$. 
Define $\overline{\rho}^* _n: [0,1] \to [0, \infty)$ by \\
$$\overline{\rho}^* _n (\alpha)=\rho ^{\frac{1}{2}} \left( \left(T_{\alpha}(K_{11}, H_{11} ) + \cdots + T_{\alpha} (K_{1m},H_{1m} )\right)  \cdots  \left(T_{\alpha}(K_{n1},H_{n1} ) + \cdots + 
T_{\alpha}(K_{nm}, H_{nm} )\right)\right)\times $$
$$\rho ^{\frac{1}{2}} \left(\left(T_{\alpha}( H_{n1},K_{n1} ) + \cdots + 
T_{\alpha}( H_{nm}, K_{nm} )\right) \cdots \left(T_{\alpha}( H_{11}, K_{11} ) + \cdots + T_{\alpha} ( H_{1m},K_{1m} )\right)  \right).$$
 Then
$\overline{\rho}^* _n$ is decreasing on $[0,\frac{1}{2}]$ and increasing on $[\frac{1}{2}, 1]$. In particular,  
 $\overline{\rho}^* _n (\alpha) \ge\overline{\rho}^* _n (\frac{1}{2})$ for all $\alpha \in [0,1]$.

\label{cor_simpler_sums_op}
\end{corollary}

\noindent {\bf Acknowledgements.} The first author acknowledges a partial support of COST Short Term Scientific Mission program (action CA18232) and the Slovenian Research and Innovation Agency (grants P1-0222 and P1-0288). The first author thanks the colleagues and staff at the Faculty of Mechanical Engineering, University of Ljubljana and at the Institute of Mathematics, Physics and Mechanics for their hospitality during her research stays in Slovenia. 

 The second author acknowledges a partial support of  the Slovenian Research and Innovation Agency (grants P1-0222 and J2-2512). The second author is also a member of the University of Ljubljana, KTT21 project
	group (grant SN-ZRD/22-27/0510).  

This arxiv preprint is a replacement of arxiv preprints of articles \cite{BP24a} and \cite{BP25}, since our main results here extend the main results of both articles.

\end{document}